\documentclass[runningheads]{llncs}
\usepackage{amsmath}
\usepackage[T1]{fontenc}

\usepackage{graphicx}
\usepackage{hyperref}
\usepackage{color}

\usepackage[utf8]{inputenc}
\usepackage{authblk}

\usepackage[caption=false]{subfig}

\usepackage{amsthm}
\usepackage{dsfont}
\usepackage{algpseudocode}
\usepackage{tikz}
\usepackage{pgfplots}
\pgfplotsset{compat=1.18}
\usetikzlibrary{calc}
\usepackage{cite}
\usepackage{lineno}
\usepackage{graphicx}
\usepackage{amssymb}
\usepackage{multirow}
\usepackage{multicol}
\usepackage{booktabs} 
\usepackage{supertabular}
\usepackage{url}
\usepackage{bm}
\usepackage{amssymb}
\usepackage[ruled,vlined]{algorithm2e}
\NoCaptionOfAlgo

\usepackage{mwe}

\newlength{\tempdima}
\newcommand{\rowname}[1]
{\rotatebox{90}{\makebox[\tempdima][c]{#1}}}

\usepackage{array}
\newcolumntype{L}{>{\centering\arraybackslash}m{2cm}}

\newcommand{\conv}{\mathop{\mathrm{conv}}}

\newtheorem{prop}{Proposition}

\newcommand{\iL}{\mathcal{L}}
\newcommand{\sstep}{\eta}

\renewcommand{\u}{{\mathbf{u}}}          
\newcommand{\e}{{\mathbf{e}}}            
\newcommand{\bv}{{\mathbf{v}}}           
\newcommand{\bu}{{\bm{u}}}               
\newcommand{\bw}{{\mathbf{w}}}           
\newcommand{\bb}{{\mathbf{b}}}           
\newcommand{\bd}{{\mathbf{d}}}           
\newcommand{\zero}{\mathbf{0}}           

\newcommand{\R}{\mathbb{R}}

\begin{document}

\title{Hybrid subgradient and simulated annealing method for hemivariational inequalities}

\author{Piotr~Bartman-Szwarc\inst{1,2}\orcidID{0000-0003-0265-6428} \and
Adil~M.~Bagirov\inst{3}\orcidID{0000-0003-2075-1699} \and Anna~Ochal\inst{1}\orcidID{0000-0003-4385-2302}}
\authorrunning{P. Bartman-Szwarc et al.}
%
\titlerunning{Hybrid subgradient and simulated annealing method for HVIs}
\institute{Chair of Optimization and Control, Jagiellonian University, Krakow, Poland
\and
Doctoral School of Exact and Natural Sciences, Jagiellonian University, Krakow, Poland
\and
Centre for Smart Analytics, Institute of Innovation, Science and Sustainability, Federation University Australia, Ballarat, Victoria, Australia}

\maketitle

\begin{abstract}
In this paper, we employ a global aggregate subgradient method for the numerical solution of hemivariational inequality problems arising in contact mechanics. The method integrates a global search procedure to identify starting points for a local minimization algorithm. The algorithm consists of two types of steps: null steps and serious steps. In each null step, only two subgradients are utilized: the aggregate subgradient and the subgradient computed at the current iteration point, which together determine the search direction. Furthermore, we compare the performance of the proposed method with selected solvers using a~representative contact mechanics problem as a case study.

\end{abstract}


\section{Introduction} \label{intro}
Hemivariational inequalities are generalizations of variational inequalities. In most cases hemivariational inequalities can be reformulated as substationary point problems of the corresponding nonsmooth nonconvex energy functions. The theory and some applications of hemivariational inequalities can be found in \cite{Panagiotopoulos1993,Panagiotopoulos1995}. To date, various methods have been developed for solving hemi\-va\-ria\-tio\-nal inequalities in different applications including in contact mechanics \cite{Han2019, Han2022, Wang21}. A~nonsmooth optimization approach to such problems is studied in \cite{Ochal2021,Makela99}.

The aim of this paper is to develop a nonsmooth nonconvex optimization method for the numerical solution of hemivariational inequalities. The proposed method is a combination of the local search subgradient method and a global search simulated annealing method. The use of the subgradient method allows to address nonsmoothness of the problem and the use of the simulated annealing method allows to deal with the nonconvexity of this problem. More specifically, we apply the subgradient method to find stationary points of the so-called energy function and then use the simulated annealing method to escape from these stationary points to find a better starting point for the local search method.

The paper is structured as follows. In Section \ref{aggsub}, first, we introduce the subgradient method for solving hemivariational inequalities and study its convergence. Then we describe the hybrid method. The application of the hybrid method for solving the contact mechanics problems is discussed in Section \ref{contactmech}. Numerical results are presented in Section \ref{numeric}. Section \ref{concl} provides some concluding remarks.

\section{A subgradient method for hemivariational inequalities} \label{aggsub}
We consider the following minimization problem
\begin{align} \label{mainprob3}
  \begin{cases}
    \text{minimize}\quad  & \iL(\u)\\
    \text{subject to}     & \u \in \R^n,
  \end{cases}
\end{align}
where
\begin{equation} \label{objefun}
\iL(\u) = \frac{1}{2} \langle A\u, \u \rangle + \langle \bb, \u  \rangle + J(\u).
\end{equation}
Here $A \in \R^{n\times n}$ is a symmetric, positively defined matrix, $\bb \in \R^n$ is a vector and $J\colon \R^n \rightarrow \R$ is a locally Lipschitz function, in general, nonsmooth nonconvex.
In what follows we denote by $\mathbb{R}^n$ the $n$-dimensional Euclidean space, $\langle \u, \bw \rangle = \sum_{i=1}^n u_iw_i$ is the inner product of vectors $\u, \bv \in \mathbb{R}^n$ and $\|\u\| = \langle \u, \u \rangle^{1/2}$ is the associated norm. $B_\epsilon(\u) = \{\bw \in \mathbb{R}^n: \|\u-\bw\| < \epsilon\}$ is an open ball centered at $\u$ with the radius $\epsilon>0$. $S_1$ is the sphere of the unit ball in $\mathbb{R}^n$.

In addition, we assume that for any $\u \in \mathbb{R}^n$ and $\bd \in S_1$ function $\iL$ satisfies
\begin{equation} \label{secant}
    \iL(\u+\tau \bd) - \iL(\u) \leq \tau \langle \bv, \bd   \rangle, \quad \bv \in \partial \iL(\u+\tau \bd), \quad \tau > 0.
\end{equation}
\noindent
Here and below $\partial \iL(\bv)$ denotes the Clarke subdifferential of $\iL$ at point $\bv$ \cite{Panagiotopoulos1995}. Moreover, let us remark that the above assumption \eqref{secant} is satisfied when $J$ is for instance a difference-of-convex function \cite{Bagirov2021}.
The objective function $\iL$ in problem~\eqref{mainprob3} is represented as the sum of three terms. The third term in this representation is a nonconvex and nonsmooth function. This makes the problem nonconvex, having many local minimizers. Therefore, we propose an algorithm consisting of two phases. In the first phase, using the current starting point, we apply the aggregate subgradient method to find a local minimizer of this problem, and then we apply a special procedure to escape from this local minimizer and find a better starting point for the aggregate subgradient method. Different versions of the subgradient and aggregate subgradient methods can be found in \cite{Bagirov2013,BagKarMak2014,Bagirov2021}. 

The method proceeds as follows. First, we choose tolerances $\varepsilon > 0, ~\delta > 0$ and three constants $\gamma \in (0,1), c_1 \in (0, 1)$ and $c_2 \in (0, c_1)$.

\begin{algorithm} [H]
    \caption{\textbf{Algorithm 1:} Solving problem \eqref{mainprob3} for a given starting point $\u$} \label{alg:combined}
    \smallskip
 \begin{algorithmic}[1]
    \State\label{alg:dir:s0} [Initialization of outer iteration] Set $l \leftarrow 1$ and $\eta \leftarrow 1$.
    \State\label{alg:dir:s1} [Initialization of inner iteration] Set $k \leftarrow 1$, select any direction $\bd_k \in S_1$ and compute $\bv_k \in \partial \iL(\u+\eta \bd_k)$. Set $\widetilde{\bv}_k \leftarrow \bv_k$.
    \smallskip
    \State\label{alg:dir:s2} Solve the following problem
     $$
     \mbox{minimize}~\varphi_k(\lambda) \equiv\|\lambda \bv_k+(1-\lambda) \widetilde{\bv}_k\|^2~\mbox{subject to}~ \lambda \in [0,1].
     $$
     Let $\lambda_k$ be a solution to this problem. Set
     $$
       \bar{\bv}_k \leftarrow \lambda_k \bv_k + (1-\lambda_k) \widetilde{\bv}_k.
     $$
    \State\label{alg:dir:s3} [Switching phase] If
      \begin{equation} \label{alg:dir:stop1}
        \|\bar{\bv}_k\| \leq \delta
      \end{equation}
      then go to Step \ref{alg:dir:s7}.
    \smallskip
    \State\label{alg:dir:s4} Compute the search direction by ${\bd}_{k+1} \leftarrow -\|\bar{\bv}_k\|^{-1} \bar{\bv}_k$.
    \smallskip
    \State\label{alg:dir:s5} If
     \begin{equation} \label{alg:dir:stop2}
       \iL (\u + \sstep {\bd}_{k+1}) - \iL (\u) \leq - c_1 \sstep \|\bar{\bv}_k\|,
     \end{equation}
     then go to Step \ref{alg:dir:s8}.
    \smallskip
    \State\label{alg:dir:s6} Compute $\bv_{k+1} \in \partial \iL(\u+\sstep {\bd}_{k+1})$. Set $~\widetilde{\bv}_{k+1} \leftarrow \bar{\bv}_k$, $k \leftarrow k+1$ and go to Step \ref{alg:dir:s2}.
    \State\label{alg:dir:s8} Compute $\u \leftarrow \u+\sigma_l {\bd}_{k+1}$, where $\sigma_l$ is defined as follows
     $$
       \sigma_l = \max \Big\{\sigma \geq \eta: ~\iL(\u+\sigma {\bd}_{k+1}) -\iL(\u) \leq -c_2\sigma \|\bar{\bv}_k\| \Big\}.
     $$
     Set $l \leftarrow l+1$ and go to Step \ref{alg:dir:s1}.
    \State\label{alg:dir:s7} Set $\eta \leftarrow \gamma \eta$. If $\eta < \varepsilon$ then STOP. Otherwise go to Step \ref{alg:dir:s1}.
  \end{algorithmic}
\end{algorithm}

\medskip
Algorithm \ref{alg:combined} consists of two loops: inner and outer loops. For a given value of $\eta > 0$ the search directions are calculated in the inner loop (Steps 2-7). The new iteration is calculated and also the parameter $\eta$ is updated in the outer loop. First, we prove that for any fixed $\eta > 0$ the inner loop terminates after finite number of iterations.

\medskip

\begin{prop} Suppose that $\iL\colon \R^n\to\R$ is a locally Lipschitz function, $\u \in \R^n$, $\sstep > 0$ and the constant $C_1 < +\infty$ is such that
\begin{equation}
C_1 = \max \left\{\|\bv\|: \bv \in \partial \iL(\u+\sstep \bd), \bd \in S_1 \right\}. \label{bounded01}
\end{equation}
If $c_1 \in (0,1)$ and $\delta \in (0,C_1)$, then the inner loop in Algorithm \ref{alg:combined} terminates after finite many iterations $m > 0$, where
$$
m \leq 2 \log_2 (\delta/C_1)/ \log_2 C_2 + 1, ~~C_2 = 1 - [(1-c_1)(2C_1)^{-1} \delta]^2.
$$
\label{algorithmdescent}
\end{prop}
\begin{proof}
The inner loop in Algorithm \ref{alg:combined} will terminate the search for a descent direction only when either condition (\ref{alg:dir:stop1}) or (\ref{alg:dir:stop2}) is satisfied. To prove that the search for the descent direction concludes in at most \( m \) steps, it suffices to establish an upper bound on the number of steps where condition (\ref{alg:dir:stop1}) is satisfied.
The proof is conducted in three stages:
first, we demonstrate that each step identifies a new subgradient $\bv_{k+1}$ that is not a convex combination of \( \bv_{k} \) and \( \bar{\bv}_k \);
next, we show that this subgradient is the best fit in terms of minimizing the corresponding functional \( \varphi_k \), i.e., \( \varphi_{k+1}(\lambda_{k+1}) < \varphi_{k}(\lambda_{k}) \);
and finally, we establish that the norm of the aggregated subgradient becomes smaller than any fixed~\( \delta \) within a finite number of steps. Let us assume, therefore, that none of the stopping conditions are met; thus, we have
$$
\mathcal{L}\left(\u+\sstep \bd_{k+1}\right) - \mathcal{L}\left(\u\right) > -c_{1} \sstep\left\|\bar{\bv}_{k}\right\| .
$$
\noindent It follows from \eqref{secant} that
$$
\mathcal{L}\left(\u+\sstep \bd_{k+1}\right) - \mathcal{L}\left(\u\right) \leq \sstep\left\langle\bv_{k+1}, \bd_{k+1}\right\rangle,
$$
\noindent so we have
$$
-c_{1} \sstep\left\|\bar{\bv}_{k}\right\| < \sstep\left\langle\bv_{k+1}, \bd_{k+1}\right\rangle,
$$
\noindent which can be simplified by using definition of $\bd_{k+1}$ to

\begin{equation} \label{proof1:eq1}
\left\langle\bv_{k+1}, \bar{\bv}_{k}\right\rangle < c_{1} \left\|\bar{\bv}_{k}\right\|^2.
\end{equation}
\noindent Moreover, since $\bar{\bv}_{k}$ is the minimal value of $\varphi_k(\cdot)$ it can be shown that

\begin{equation} \label{proof1:eq2}
\left\langle\lambda \bv_k+(1-\lambda) \widetilde{\bv}_k, \bar{\bv}_{k}\right\rangle \geq c_{1} \left\|\bar{\bv}_{k}\right\|^2, \quad
\text{for all} \,\, \lambda \in [0, 1].
\end{equation}

\noindent Thus, based on equations \eqref{proof1:eq1} and \eqref{proof1:eq2}, we can deduce that
$$
\text{there is no }
\lambda \in [0, 1] \text{ such that } \bv_{k+1} = \lambda \bv_k+(1-\lambda) \widetilde{\bv}_k,
$$
\noindent which is what we aimed to demonstrate in the first step. Using the definition of $\bar\bv_{k+1}$ and the fact that $\lambda_{k+1}$ minimizes $\varphi_{k+1}(\cdot)$, we can write

\begin{align*}
\|\bar\bv_{k+1}\|^2
& =
\| \lambda_{k+1} \bv_{k+1} + (1 - \lambda_{k+1} ) \widetilde{\bv}_{k+1}\|^2
\\ & \leq
\| \lambda \bv_{k+1} + (1 - \lambda ) \widetilde{\bv}_{k+1}\|^2 \quad \forall \lambda \in [0, 1],
\end{align*}

\noindent which is equivalent to
$$
\|\bar\bv_{k+1}\|^2 \leq \lambda^2 \| \bv_{k+1} - \widetilde{\bv}_{k+1}\|^2 + \|\widetilde{\bv}_{k+1}\|^2 + 2 \lambda \left\langle \bv_{k+1} - \widetilde{\bv}_{k+1},  \widetilde{\bv}_{k+1}\right\rangle \quad \forall \lambda \in [0, 1].
$$

From the boundedness assumption (\ref{bounded01}), we know that $\| \bv_{k+1} - \widetilde{\bv}_{k+1}\| \leq 2 C_1$, and furthermore, using (\ref{proof1:eq1}) for $k > 1$, we obtain
$$
\left\langle \bv_{k+1} - \widetilde{\bv}_{k+1},  \widetilde{\bv}_{k+1}\right\rangle
=
\left\langle \bv_{k+1},  \bar\bv_{k}\right\rangle - \left\langle \bar\bv_{k},  \bar\bv_{k}\right\rangle
\leq
(c_{1} - 1) \left\|\bar{\bv}_{k}\right\|^2.$$
Thus, we derive
\begin{equation} \label{eq:vnorm}
\|\bar\bv_{k+1}\|^2
\leq
(2 \lambda C_1)^2 + (1 + 2 \lambda (c_{1} - 1) \left\|\bar{\bv}_{k}\right\|^2) \|\bar{\bv}_{k}\|^2 \quad \forall \lambda \in [0, 1].
\end{equation}
Choosing arbitrary $\lambda_0 \in [0, 1]$
$$
\lambda_0 = \frac{(1 - c_1) \|\bar{\bv}_{k}\|^2}{(2C_1)^2},
$$
and substituting it into inequality \eqref{eq:vnorm} we get
$$
\|\bar\bv_{k+1}\|^2
\leq
\left(1 - \left(\frac{(1 - c_1) \|\bar{\bv}_{k}\|}{2C_1}\right)^2\right) \|\bar{\bv}_{k}\|^2.
$$
From assumption that (\ref{alg:dir:stop1}) is not satisfied we know that $\delta < \|\bar{\bv}_{k}\|$. Denoting $C_2 = 1 - ((1 - c_1) \delta)^2 (2C_1)^{-2}$ and using the boundedness assumption (\ref{bounded01}) for $\|\bar{\bv}_{k}\|^2$ it follows
$$
\|\bar\bv_{k}\|^2
<
C_1^2 C_2^k.
$$
From above and the fact that $C_2 \in (0, 1)$ we deduce that (\ref{alg:dir:stop1}) is satisfied if
$$
\delta^2
\geq
C_1^2 C_2^k,
$$
what is true if $k = m$.
\end{proof}

For a given $\eta > 0$ consider the following convex hull of the set of subgradients
$$
W_\eta(\bu) = \conv \Big\{\bv \in \mathbb{R}^n:~~\bv \in \partial \iL(\bu+\eta \bd),~\bd \in S_1 \Big\}.
$$
Let $\delta > 0$ be given. A point $\bar{\bu}$ is called an $(\eta,\delta)$-stationary point of the function~$\iL$ iff
$$
\zero \in W_\eta(\bar{\bu}) + B_\delta(\zero).
$$

\begin{prop} Suppose that function $\iL$ is bounded below
\begin{equation}
\iL^*=\inf ~\{\iL(\u): \u \in \R^n\} > -\infty. \label{lowbound}
\end{equation}
Let $\u_0\in\R^n$ be an initial point. Then Algorithm \ref{alg:combined} terminates after finite many iterations $M > 0$ and produces ($\sstep,\delta$)-stationary point $\u_M$ where
\begin{equation} \label{upperbound}
M \leq M_0 \equiv \left\lfloor \frac{\iL(\u_0)-\iL^*}{c_2\eta\delta} \right\rfloor + 1.
\end{equation}
\label{hdelta1}
\end{prop}
\begin{proof}
We conduct proof by contradiction. Let us assume that the sequence $\{\u_l\}$ generated by Algorithm \ref{alg:combined} is infinite and for any $l \in \mathbb{N}^+$ the point $\u_l$ is not $(\sstep, \delta)$-stationary point, i.e.
$$
\min \{\|\bv\| : \bv \in W_\sstep(\u_l)\} > \delta, \quad \forall l \in \mathbb{N}^+.
$$
Then we know that the descent direction ${d}_{k+1}$ can be found so that
$$
\iL(\u_l + \sstep {\bd}_{k+1}) - \iL (\u_l) < - c_1 \sstep \|\bar{\bv}_k\| \leq - c_2 \sstep \|\bar{\bv}_k\|
$$
It follows from the definition that $\sigma_l \geq \sstep$. Therefore, we get
$$
\iL(\u_{l+1}) - \iL(\u_l) = \iL(\u_l + \sstep {\bd}_{k+1}) - \iL (\u_l) < - c_2 \sigma_l \|\bar{\bv}_k\| \leq -c_2 \sstep \|\bar{\bv}_k\|,
$$
in addition the condition $\|\bar\bv_{k}\| > \delta$ is satisfied that implies
$$
\iL(\u_{l+1}) \leq \iL(\u_0) - (l + 1)c_2 \eta \delta.
$$
Therefore, $\iL(\u_l) \to -\infty$ as $l \to \infty$, which contradicts (\ref{lowbound}). Clearly, the upper bound for the number of iterations $M$ necessary to find the $(\sstep,\delta)$-stationary point is $M_0$ given by (\ref{upperbound}).
\end{proof}

To describe the hybrid algorithm we will use the Metropolis function

 $$   R(\bw,\u,T) = \min \Big\{1, \exp( {(} \iL(\u) - \iL(\bw) {)} /T) \Big\} $$

where $\u,\bw \in \mathbb{R}^n$ and $T >0$. Let $\e_i \in \mathbb{R}^n$ be the $i$-th standard unit vector.

\begin{algorithm}
    \caption{\textbf{Algorithm 2:} Hybrid subgradient and simulated annealing method for solving problem \eqref{mainprob3}} \label{alg:glb}
    \smallskip
 \begin{algorithmic}[1]
    \State\label{alg:hyb:s0} [Initialization] Select the initial point $\u_0 \in \mathbb{R}^n$, the initial temperature $T_0 \in (1, \infty)$, the minimum temperature $T_{min} < T_0$, the temperature reduction factor $\alpha \in (0,1)$. Set $\u_{best} \leftarrow \u_0, \iL_{best} \leftarrow \iL(\u_0),~\bar{\u} \leftarrow \u_0$ and $k \leftarrow 0$.
    \State\label{alg:hyb:s1} Apply Algorithm \ref{alg:combined} starting from the point $\bar{\u}$ and find the stationary point of problem \eqref{mainprob3}. Denote it by $\u_{k+1}$. If $\iL(\u_{k+1}) < \iL_{best}$ then update $\u_{best} \leftarrow \u_{k+1}, \iL_{best} \leftarrow \iL(\u_{k+1})$.
    \smallskip
    \State\label{alg:hyb:s2} Generate a uniformly distributed random number $\mu$ from $[0,1]$, 
    randomly select $i \in \{1,\ldots,n\}$ and calculate a~trial point $\bar{\bw} \leftarrow \u_{k+1}+\mu \e_i$.
    \smallskip
    \State\label{alg:hyb:s3} If $\iL(\bar{\bw}) < \iL_{best}$\,, then update $\u_{best} \leftarrow \bar{\bw}$, $\iL_{best} \leftarrow \iL(\bar{\bw})$, set $\bar{\u} \leftarrow \bar{\bw},~k \leftarrow k+1$ and go to Step \ref{alg:hyb:s1}.
    \smallskip
    \State\label{alg:hyb:s4} Sample a uniformly distributed random number $\beta$ from $[0,1]$. If $\beta \leq R(\bar{\bw}, \u_{k+1},T)$, then set $\bar{\u} \leftarrow \bar{\bw}$ and go to Step \ref{alg:hyb:s1}.
    \smallskip
    \State\label{alg:hyb:s5} Set $T\leftarrow\alpha T$. If $T < T_{min}$ then STOP. $\u_{best}$ is a solution. Otherwise go to Step \ref{alg:hyb:s2}.
  \end{algorithmic}
\end{algorithm}
\begin{remark}
    Algorithm \ref{alg:glb} is based on the combination of Algorithm \ref{alg:combined} and the si\-mu\-la\-ted annealing method. We apply Algorithm \ref{alg:combined} to find the stationary point of the function $\iL$ and then apply the simulated annealing method to escape from this point and find a new starting point for Algorithm \ref{alg:combined}. The convergence of the simulated annealing method is studied, for example, in \cite{Locatelli2000}. In the proposed hybrid algorithm, the simulated annealing method is only used to find starting points for the local search algorithm. This means that the hybrid method developed in this paper converges to the global minimizer of the function $\iL$ with probability one.
\end{remark}

\section{Contact Mechanics Problem} \label{contactmech}

In this section, we present an example of contact mechanics problem with the relevant physical context and notations.

An elastic body is considered to occupy a domain \(\Omega \subset \mathbb{R}^d\), where \(d = 2, 3\) in  applications.
The boundary \(\Gamma\) of the domain is divided into three measurable parts: \(\Gamma_D\), \(\Gamma_C\), and \(\Gamma_N\), with \(\Gamma_D\) having positive measure.
The boundary \(\Gamma\) is Lipschitz continuous ensuring the existence of the outward normal vector \(\bm{\nu}\) almost everywhere on \(\Gamma\). Boundary conditions specify that the body is clamped on \(\Gamma_D\) meaning the displacement satisfies \(\bu = \bm{0}\) there.
A surface force with density \(\bm{f}_N\) acts on \(\Gamma_N\), while a body force density \(\bm{f}_0\) is applied throughout \(\Omega\).
The contact interaction on \(\Gamma_C\) are govern by using general subdifferential inclusions. For the sake of simplicity we consider frictionless case.
The objective is to determine the displacement of the body in static equilibrium.

The scalar product and Euclidean norm in \(\mathbb{R}^d\) or \(\mathbb{S}^d\) (the space of symmetric second-order tensors) are denoted by ``\(\cdot\)'' and \(\|\cdot\|\), respectively.
The normal and tangential components of displacement \(\bu\) and stress \(\bm{\sigma}\) on \(\Gamma_C\) are represented by \(u_\nu = \bu\cdot\bm{\nu}\), \(\bu_\tau = \bu - u_\nu\bm{\nu}\), \(\sigma_\nu = \bm{\sigma}\bm{\nu}\cdot\bm{\nu}\), and \(\bm{\sigma}_\tau = \bm{\sigma}\bm{\nu} - \sigma_\nu\bm{\nu}\), respectively.
The small strain tensor is defined as \(\bm{\varepsilon}(\bu) = (\varepsilon_{ij}(\bu))\), where:
\(
\varepsilon_{ij}(\bu) = \frac{1}{2}(\frac{\partial u_i}{\partial x_j} + \frac{\partial u_j}{\partial x_i}).
\)

\

\noindent
\textbf{Problem $P$:} Find a displacement  \(\bu\colon \Omega \to \mathbb{R}^d\) and a stress \(\bm{\sigma}\colon \Omega \to \mathbb{S}^d\) satisfying
\begin{align}
  \bm{\sigma}  = \mathcal{A}(\bm{\varepsilon}(\bm{u})) \qquad &\textrm{ in } \Omega, \label{P1}\\
  \textrm{Div }\bm{\sigma} + \bm{f}_{0} = \bm{0}  \qquad &\textrm{ in } \Omega, \label{P2}\\
  \bm{u} = \bm{0}  \qquad &\textrm{ on } \Gamma_{D}, \label{P3}\\
  \bm{\sigma}\bm{\nu} = \bm{f}_{N} \qquad &\textrm{ on } \Gamma_{N}, \label{P4}\\
   -\sigma_{\nu} \in \partial j(u_{\nu}) \qquad &\textrm{ on } \Gamma_{C}, \label{P5}\\
   \bm{\sigma}_\tau = 0 \qquad &\textrm{ on } \Gamma_{C}. \label{P6}
\end{align}

\noindent
Here, equation~\eqref{P1} represents an elastic constitutive law and $\mathcal{A}$ is an elasticity operator.
Equilibrium equation~(\ref{P2}) reflects the fact that problem is static.
Equation~(\ref{P3}) represents clamped boundary condition on $\Gamma_{D}$ and~(\ref{P4}) represents the action of the traction on $\Gamma_{N}$.
Inclusion~(\ref{P5}) describes the response of the foundation in normal direction, where $j$ is a given potential.
Equation~(\ref{P6}) means that contact is frictionless.

The Hilbert spaces for the problem are
\[
\mathcal{H} = L^2(\Omega; \mathbb{S}^d), \quad V = \{\bm{v} \in H^1(\Omega)^d \mid \bm{v} = \bm{0} \text{ on } \Gamma_D\},
\]
the latter with norm defined through strain tensors and Korn's inequality ensuring completeness. The trace operator \(\gamma\colon V \to L^2(\Gamma_C)^d\) is continuous by the Sobolev trace theorem.

Using the Green formula and the definition of generalized subdifferential, a~weak formulation of Problem $P$ is derived as a hemivariational inequality.

\medskip

\noindent
\textbf{Problem $P_{hvi}$:}  Find \(\bu \in V\) such that
\[
(\mathcal{A}(\bm{\varepsilon}(\bm{u})),\bm{\varepsilon}(\bm{v}))_{\mathcal{H}} + \int_{\Gamma_C} j^0(u_\nu; v_\nu)\,da \geq \langle \bm{f}, \bm{v} \rangle_{V^*\times V}\quad \text{for all }\bm{v} \in V,
\]

\noindent
where for $\bm{f}_0 \in L^2(\Omega)^d, \quad \bm{f}_N \in L^2(\Gamma_N)^d$
\begin{align*}
 \langle \bm{f}, \bm{v} \rangle_{V^* \times V} = \int_{\Omega}\bm{f}_{0}\cdot \bm{v}\, dx + \int_{\Gamma_{N}}\bm{f}_{N}\cdot \gamma \bm{v}\, da. \nonumber
\end{align*}

\noindent
Below we present the rest of necessary assumptions for existence of the solution to Problem~$P_{hvi}$ (see more \cite{Ochal2021}).

\medskip
\noindent
The elasticity operator ${\mathcal{A}} \colon \Omega \times {\mathbb S}^d \to {\mathbb S}^d$ satisfies
\begin{enumerate}
  \item[(a)]
     $\mathcal{A}(\bm{x},\bm{\tau}) = (a_{ijkh}(\bm{x})\tau_{kh})$
     for all $\bm{\tau} \in {\mathbb S}^d$, a.e. $\bm{x}\in\Omega,\ a_{ijkh} \in L^{\infty}(\Omega),$
  \item[(b)]
    $\mathcal{A}(\bm{x},\bm{\tau}_1) \cdot \bm{\tau}_2 = \bm{\tau}_1 \cdot  \mathcal{A}(\bm{x},\bm{\tau}_2)$ for all $\bm{\tau}_1, \bm{\tau}_2 \in {\mathbb S}^d$, a.e. $\bm{x}\in\Omega$,
  \item[(c)]
     $\exists \, m_{\mathcal{A}}>0$ such that $\mathcal{A}(\bm{x},\bm{\tau}) \cdot \bm{\tau} \geq m_{\mathcal{A}} \|\bm{\tau}\|^2$ for all $\bm{\tau} \in {\mathbb S}^d$, a.e. $\bm{x}\in\Omega$.
\end{enumerate}

\noindent
The potential $j \colon \Gamma_C \times \mathbb{R} \to \mathbb{R}$ satisfies
\begin{enumerate}
  \item[(a)]
    $j(\cdot, \xi)$ is measurable on $\Gamma_C$ for all $\xi \in \mathbb{R}$ and there exists $e \in L^2(\Gamma_C)$ such that $j(\cdot,e(\cdot))\in L^1(\Gamma_C)$,
  \item[(b)]
    $j(\bm{x}, \cdot)$ is locally Lipschitz continuous on $\mathbb{R}$ for a.e. $\bm{x} \in \Gamma_C$,
  \item[(c)]
    there exist $c_{0}, c_{1} \geq 0$ such that \\[2mm]
    \hspace*{1cm}$|\partial j(\bm{x}, \xi)| \leq c_{0} + c_{1}|\xi|$\quad for all $\xi \in \mathbb{R}$, a.e. $\bm{x} \in \Gamma_C$,
  \item[(d)]
    there exists $\alpha \geq 0$ such that \\[2mm]
    \hspace*{1cm}$j^0(\bm{x},\xi_1;\xi_2-\xi_1) + j^0(\bm{x},\xi_2;\xi_1-\xi_2)\leq \alpha|\xi_1-\xi_2|^2$ \\[2mm]
     for all $\xi_1, \xi_2 \in \mathbb{R}$, a.e. $\bm{x} \in \Gamma_C$.
\end{enumerate}

\noindent
We define the functional $J \colon L^2(\Gamma_C)^d \to \mathbb{R}$ by
\begin{align*}
J(\bm{v}) =  \int_{\Gamma_C} j(v_\nu)\, da\quad \text{for all } \bm{v} \in L^2(\Gamma_C)^d.\label{J}
\end{align*}

To find a solution of the discrete version of hemivariational inequality (Problem~$P_{hvi}$) we use an optimization-based method described in detail in \cite{Ochal2021}. Let $V^h \subset V$ be a finite dimensional subspace with a discretization parameter $h>0$.
The corresponding optimization problem in the case of Problem~$P_{hvi}$ is as follows.

\medskip
\noindent
\textbf{Problem ${P_{opt}^{h}}$:} {Find  $\bm{u}^h \in V^h$ that minimizes functional ${\mathfrak{L}}: V \rightarrow \mathbb{R}$ defined by
\begin{align*}
  {\mathfrak{L}}(\bm{v}) = \frac{1}{2} (\mathcal{A}(\bm{\varepsilon}(\bm{v})),\bm{\varepsilon}(\bm{v}))_{\mathcal{H}}  - \langle  \bm{f} ,  \bm{v}  \rangle_{V^*\times V} + J(\gamma  \bm{v}  )\quad \text{for all }\bm{v} \in V.
\end{align*}
}

To numerically solve the above optimization problem, we use the finite element method. We discretize the elasticity operator, the forces \( \bm{f} \), and the contact condition, which allows us to construct the functional in the form (\ref{objefun}).

\section{Numerical results} \label{numeric}

To verify the accuracy and efficiency of the proposed method, we conducted a~series of simulations.
To keep this work concise, we focus on the simplified model introduced in the previous section.
The schematic representation of the modeled 2D beam is shown in Figure~\ref{fig:ref}.
The beam thickness is 10\,mm, length 210\,mm and it is clamped at both ends and rests directly on a soft obstacle made of a~composite material.
A force is applied to the beam from above:

\[
\bm{f}_{N}(x) =
\left( 0, - L \, S  \left( 1 - \frac{(x - 105\,\text{mm})^2}{(105\,\text{mm})^2} \right) \right),
\]

\noindent
where beam upper surface $S = 1000\,\text{mm}^2$ and $L$ is applied loading.

\begin{figure}
 \includegraphics[width=\textwidth]{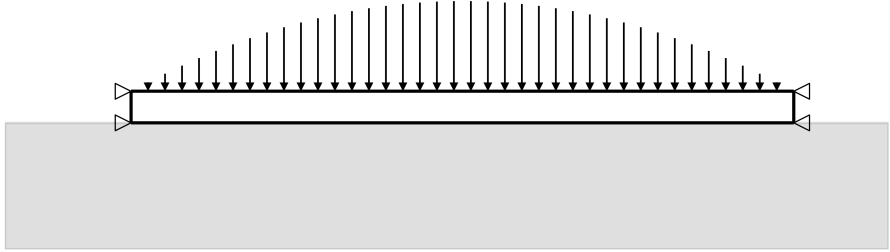}
 \caption{Modeled beam under loading.}\label{fig:ref}
\end{figure}

As the beam deflects under the applied force and penetrates the obstacle.
The composite nature of the foundation is characterized by the subgradient of the functional \( j_n \). Examples of the functional \( j_n \) for $n=2$ (black) and $n=7$  (gray) are illustrated in Figure~\ref{fig:func}.
For more examples of composite materials, see e.g. \cite{Tair18, Tair19}.
The presented contact law corresponds to a material whose reaction force increases with penetration until a critical point is reached, at which a composite layer cracks, causing a reduction in the foundation's reaction force.

The function \( j_2 \) represents two composite layers represents a soft base covered by a thin (3\,mm) protective layer, which is responsible for the initial increase in reaction force with penetration.
Once the critical penetration threshold is exceeded, the protective layer cracks, and the reaction force ceases to increase, stabilizing at a constant level.
Similarly, in the function \( j_7 \) represents seven layers, where are six progressively thicker protective layers, each contributing to the force response until crack occurs.
The total thickness of all protective layers is constant and equal to 3\,mm.
For any $n>2$ function $j_n$ is nondifferentiable and nonconvex.

\begin{figure}
 \includegraphics[width=\textwidth]{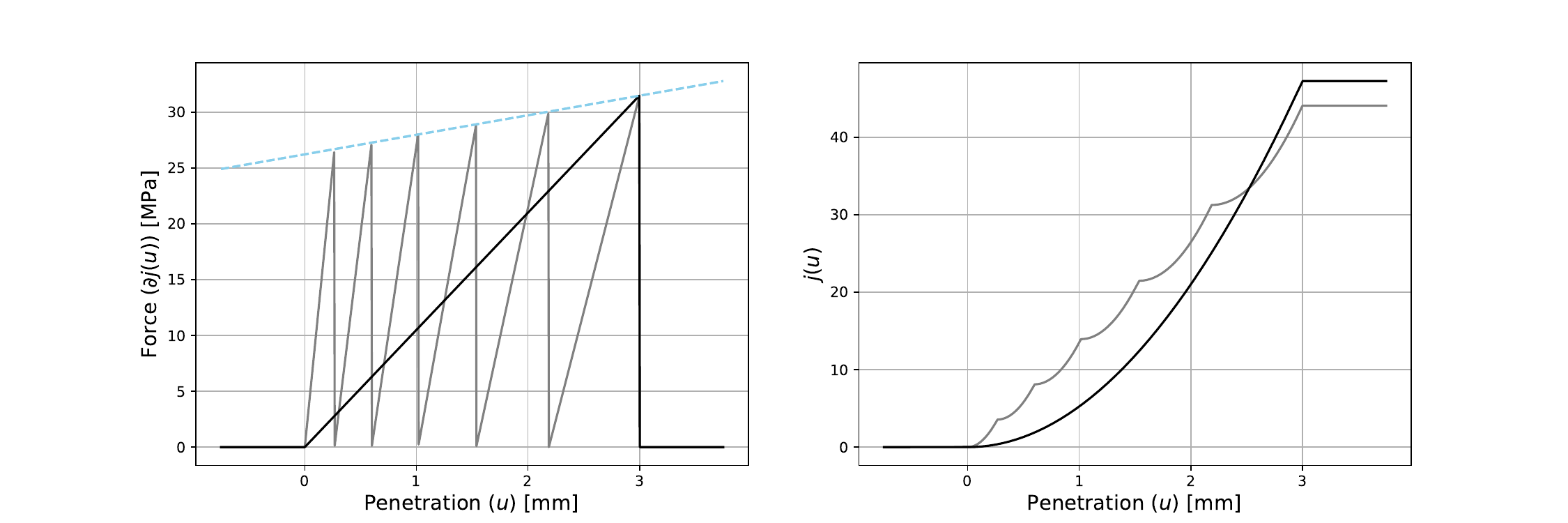}
 \caption{A nonmonotone contact law and the corresponding functional \( j(u) \).}\label{fig:func}
\end{figure}

For a sake of simplicity, we neglect the body internal forces, i.e.,  $\bm{f}_0=\bm{0}$.
Elasticity tensors are given by
\begin{eqnarray*}
({\mathcal A}\bm{\tau})_{ij}& = &\frac{E \kappa}{(1 + \kappa)(1 - 2\kappa)}(\tau_{11} + \tau_{22})\delta_{ij} \, + \, \frac{E}{1 + \kappa}\tau_{ij} \label{elastic_tensor}\\
&&\qquad\qquad\qquad\qquad\forall
\, \bm{\tau}=(\tau_{ij}) \in \mathbb{S}^2, \ i,j = 1, 2  \nonumber. \label{relaxation_tensor}
\end{eqnarray*}
\noindent
Here and below $\delta_{ij}$ is the Kronecker delta, and $E$ and $\kappa$ are Young's modulus and Poisson's ratio of the body material, respectively.
We chose $E = 9.646 \cdot 10^7 \, \text{MPa}$ and $\kappa = 0.4 $.

\begin{figure}
 \includegraphics[width=0.5\textwidth]{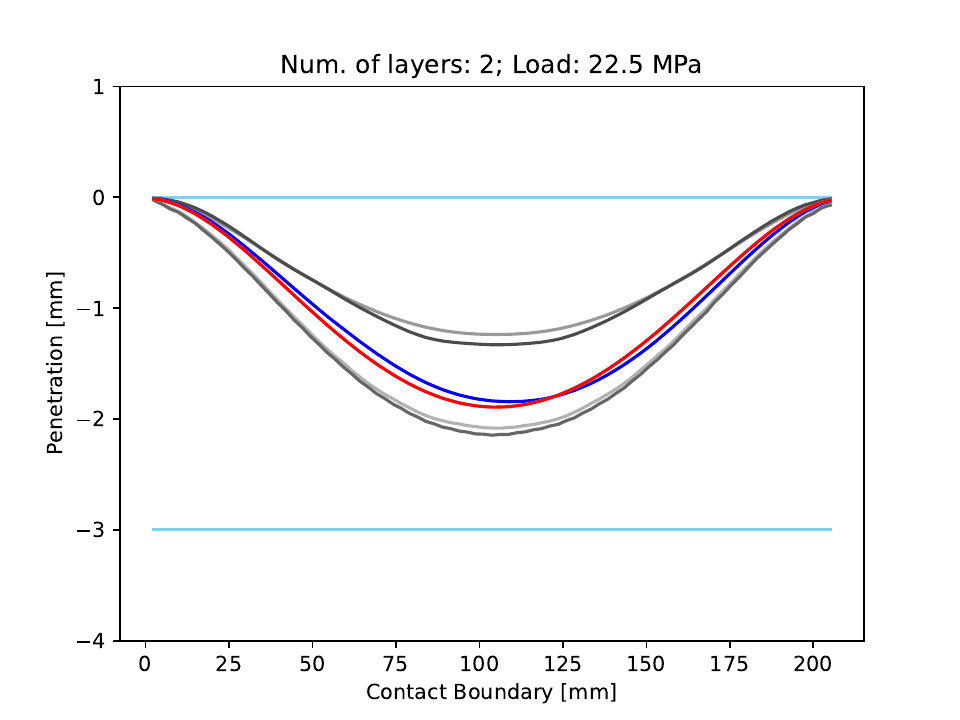}
 \includegraphics[width=0.5\textwidth]{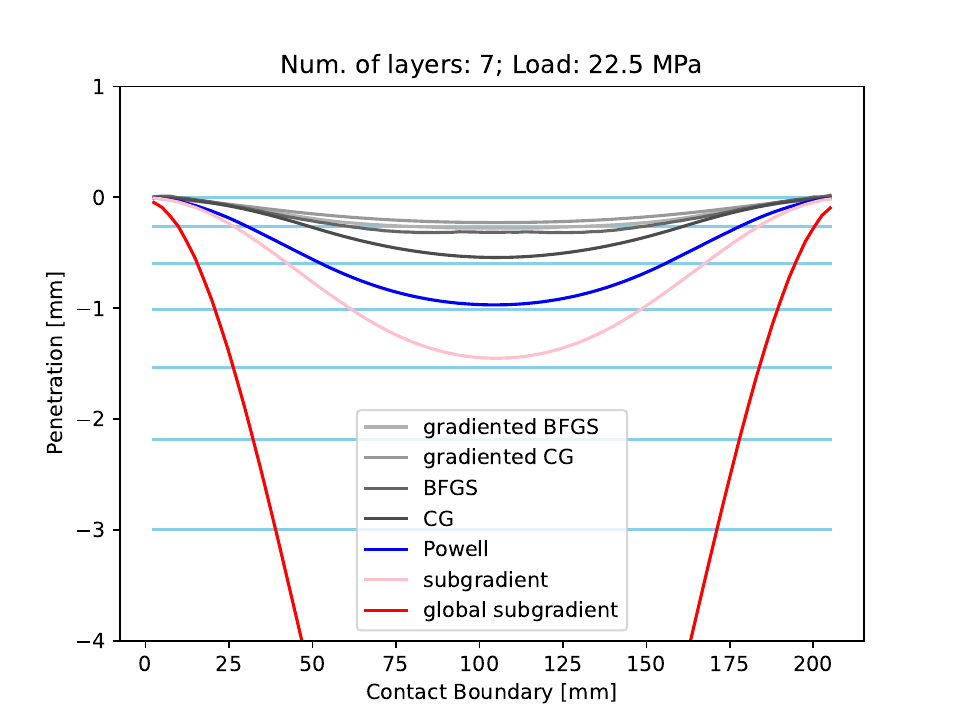}
 \caption{Deflection of the beam under loading.}\label{fig:deflect}
\end{figure}

We present the results of numerical experiments conducted for the contact law based on the functions \( j_2 \), \( j_3 \), \( j_7 \), and \( j_{10} \). For each case, we performed 10 simulations under varying load conditions. The optimization problem in each simulation was solved using six different methods which was: four optimization techniques from the \texttt{numpy} package \cite{NUMPY} and two proposed in the following paper. A mesh with a finite element size not exceeding 1.75\,mm was used.

Figure \ref{fig:deflect} highlights the increased differences between methods observed in multi-layered materials, further illustrating the complexities introduced by additional layers in the composite structure.
In the left plot, we see the deformation of the body in the case of two layers (protective and base), while the right plot shows the deformation for seven layers.

\begin{figure}[p]
 \centering
 \includegraphics[width=0.875\textwidth]{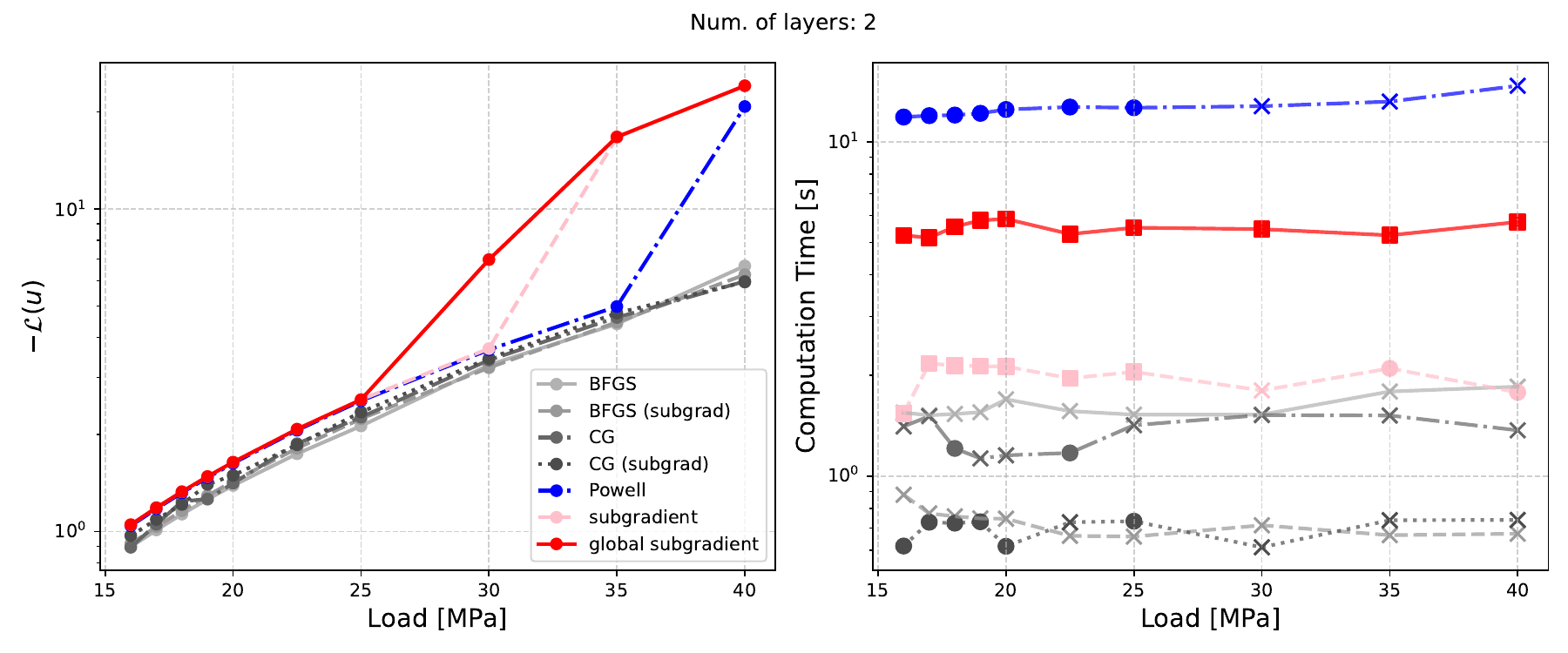}
 \includegraphics[width=0.875\textwidth]{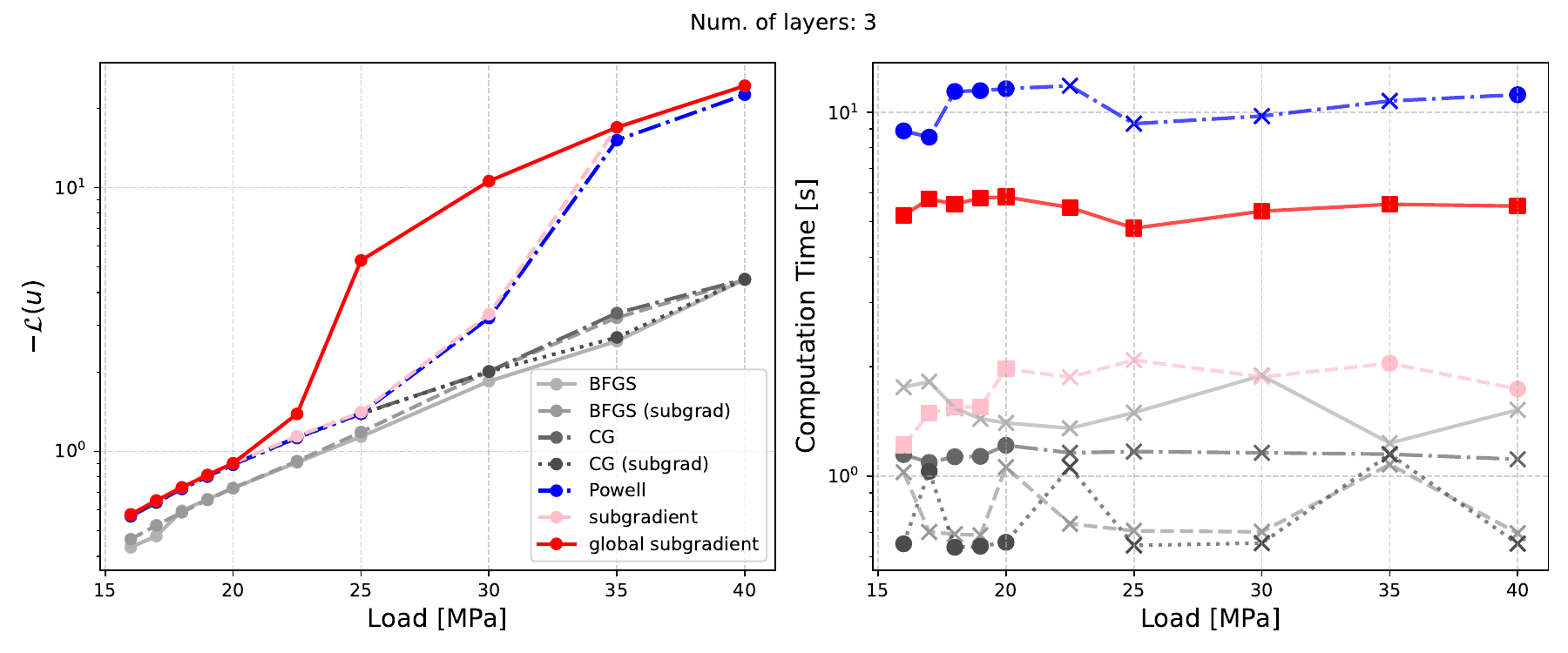}
 \includegraphics[width=0.875\textwidth]{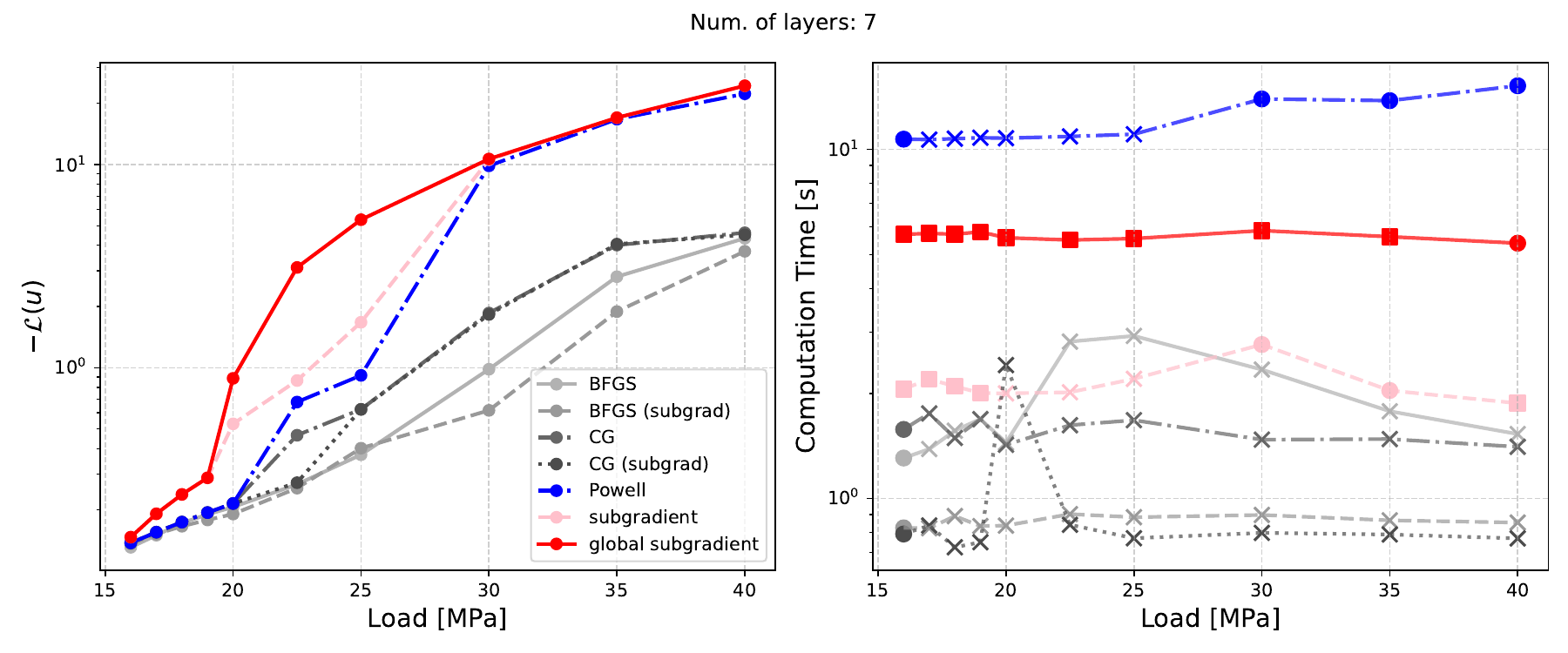}
 \includegraphics[width=0.875\textwidth]{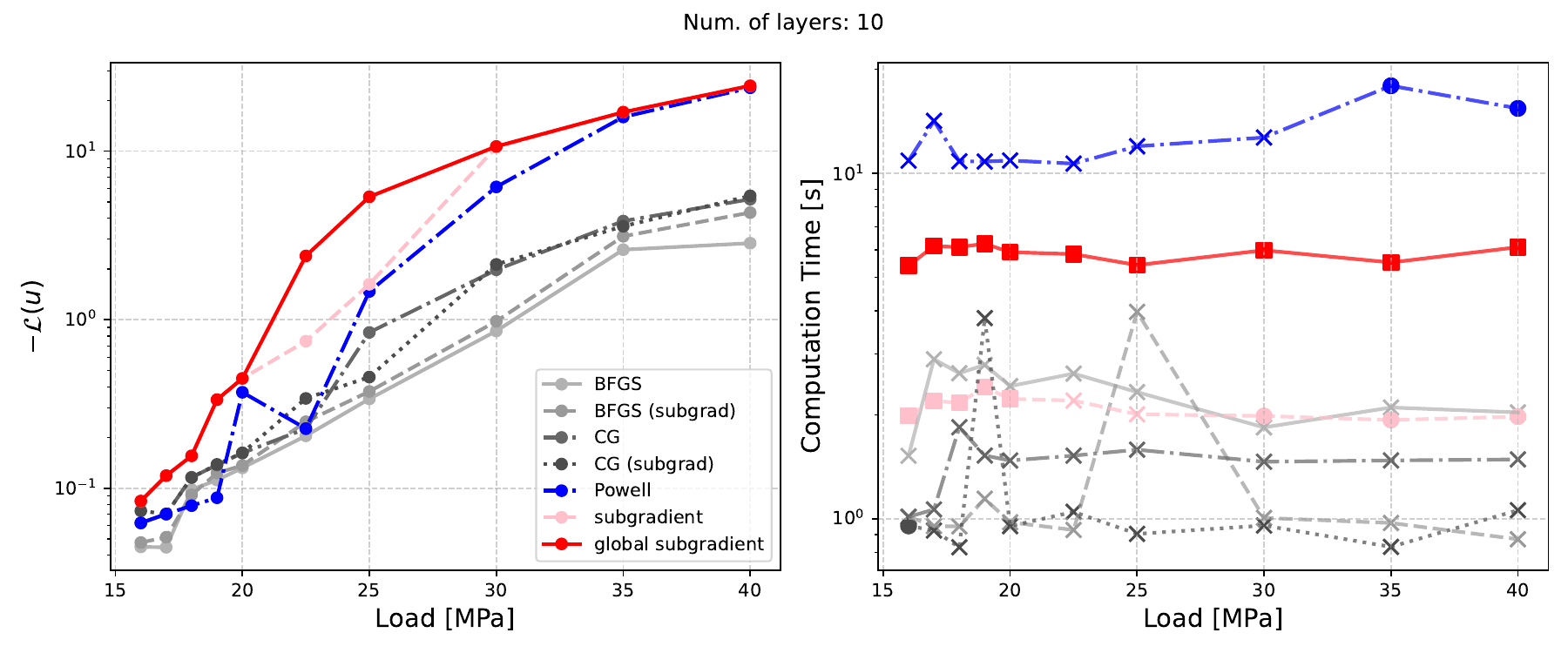}
 \caption{Energy function values for selected simulations (left column) and the corresponding computation time (right column).}\label{fig:perf}
\end{figure}

The methods - BFGS, Conjugate Gradient Method (CG), ``gradiented BFGS'', and ``gradiented Conjugate Gradient Method'' (``gradiented CG'') - exhibited short optimization times but failed to yield satisfactory results. The prefix ``gradiented'' indicates that the method explicitly utilized the computed subgradient of the cost function, as opposed to its non-prefixed counterparts, which estimated the gradient. These methods are marked in shades of gray in the presented plots.

A more notable approach is Powell's method, which demonstrated significantly better results. Although Powell's method generally exhibits high precision, it is computationally expensive. Moreover, as illustrated in the left column of Figure \ref{fig:perf}, it does not always find the optimal solution. Notably, its performance fluctuates around intermediate force values, which correspond to the progressive rupture of individual material layers - an inherently nonlinear phenomenon.

Our proposed methods, denoted as ``subgradient'' and ``global subgradient'', perform notably better when dealing with a larger number of composite material layers. Due to the possibility of multiple local minima, the basic subgradient method may occasionally become trapped. To mitigate this issue, we developed a strategy for selecting promising initial points for the subgradient method. Specifically, we employed a search algorithm that identified five starting points for each optimization run.

Note that in Figure \ref{fig:perf} the $y$-axis of the plots employs a logarithmic scale, and the cost function values are presented with a negative sign to enhance readability - thus, higher values indicate better results. The right column of Figure \ref{fig:perf} presents the computational time required for optimization. The markers distinguish three categories of results: squares indicate the lowest achieved cost function value $\iL_\text{best}$, circles represent values within section $(0.9 \, \iL_\text{best} + 0.01, \iL_\text{best})$, and ``x'' denotes results that deviate more significantly from $\iL_\text{best}$. These plots reveal that the ``global subgradient'' method outperforms the others in finding the global minimum while also being computationally more efficient than Powell’s method.

It is worth emphasizing that subgradient and global subgradient methods were implemented in Python.
We employ our original open-source package \textit{conmech} \cite{CONMECH}, a user-friendly Pythonic tool designed for contact mechanical si\-mu\-la\-tions.
The package is a comprehensive simulation framework, allowing for straightforward definition of body geometry and material properties, automatic generation of computational meshes, and empirical error analysis. It supports simulations for static, quasistatic, and dynamic problems in both two and three dimensions.
The software is designed with modularity in mind, enabling seamless extension of existing models with additional physical effects.
To improve computational robustness in \textit{conmech} we utilize the just-in-time compiler \textit{Numba} \cite{NUMBA}.
Wall times for CPU depicted in right column of Figure \ref{fig:perf} are obtained on commodity hardware: MacBook Pro M1 16GB.

\section{Conclusion} \label{concl}
In this paper, we developed a method for the numerical solution of hemiva\-ria\-tional inequalities. The hemivarional inequalities problem is reduced to the minimization of the so-called energy function. This function is both nonsmooth and nonconvex. The proposed method is a hybrid of the subgradient and the si\-mu\-la\-ted annealing methods. The subgradient is applied to find stationary points of nonsmooth energy function and the simulated annealing is used to escape from these stationary points and find better starting points for the subgradient method. In this way the proposed hybrid method can efficiently deal with both nonsmoothness and nonconvexity of the problem under consideration.

\begin{credits}
\subsubsection{\ackname} This project has received funding from the European Union’s Horizon 2020 Research and Innovation Programme under the Marie Sklodowska-Curie
Grant Agreement No 823731 CONMECH. The first and third authors are supported by National Science Center, Poland, under project OPUS no. 2021/41/B/ST1/01636.

\subsubsection{\discintname}
The authors have no competing interests to declare that are relevant to the content of this article.
\end{credits}

\bibliographystyle{splncs04}
\bibliography{Hemivarbib}

\end{document}